\documentclass[12pt]{article}
\usepackage{amsmath,amsfonts,amssymb,theorem,graphicx, verbatim}
\usepackage{hyperref}
\usepackage{pb-diagram}
\newcommand{\CC}{\mathbb C}
\newcommand{\OO}{\mathbb O}

\newcommand{\PP}{\mathbb P}
\newcommand{\ZZ}{\mathbb Z}

\newcommand{\QQ}{\mathbb Q}

\newcommand{\lam}{\lambda}

\newcommand{\Lam}{\Lambda}
\newcommand{\al}{\alpha}
\newcommand{\be}{\beta}

\newcommand{\la}{\lambda}
\newcommand{\nab}{\nabla}

\newcommand{\om}{\omega}

\newcommand{\Si}{\Sigma}

\newcommand{\Oh}{\mathcal O}

\newcommand{\sE}{\mathcal E}
\newcommand{\Cc}{\mathcal C}

\newcommand{\sS}{\mathcal S}
\newcommand{\sQ}{\mathcal Q}

\newcommand{\into}{\hookrightarrow}

\newcommand{\color}[6]{}

\theorembodyfont{\rmfamily}
 \newtheorem{theorem}[subsection]{Theorem}

 \newtheorem{prop}[subsection]{Proposition}

\newtheorem{example}[subsection]{Example}

{
\theorembodyfont{\rmfamily}

 \newtheorem{nothing}[subsection]{}
}
 
 \newenvironment{proof}{\paragraph{Proof}}{\par\medskip}

\theorembodyfont{\rmfamily}

\DeclareMathOperator{\GL}{GL}

\DeclareMathOperator{\Gr}{Gr }

\numberwithin{equation}{section}
\usepackage[normalem]{ulem}
\begin{document}

\title{Calabi--Yau threefolds in weighted flag varieties}
\date{}
\author{Muhammad Imran Qureshi and Bal{\'a}zs Szendr{\H{o}}i}

\maketitle
\begin{abstract} We review the construction of families of projective varieties, 
in particular Cala\-bi--Yau threefolds, as quasilinear sections in weighted flag varieties. 
We also describe a construction of tautological orbi-bundles on these varieties, which may be of interest in heterotic model building. 
\end{abstract}

\section{Introduction} 

The classical flag varieties $\Sigma=G/P$ are projective varieties which are 
homogeneous spaces under complex reductive Lie groups~$G$; the stabilizer $P$ 
of a point in $\Sigma$ is a parabolic subgroup $P$ of $G$. The simplest example
is projective space $\PP^n$ itself, which is a homogeneous space under the
complex Lie group $\GL(n)$. Weighted flag varieties $w\Sigma$, which are the
analogues of weighted projective space in this more general context,
were defined by Grojnowski and Corti--Reid~\cite{wg}.
They admit a Pl\"ucker-style embedding
\[w\Sigma\subset\PP[w_0, \cdots , w_n]\]
into a weighted projective space. In this paper, we review the construction of
Calabi--Yau threefolds~$X$ that arise as complete intersections within~$w\Sigma$ 
of some hypersurfaces of weighted projective space~\cite{wg, qs, Q}:
\[X\subset w\Sigma\subset\PP[w_0, \cdots , w_n].\]
To be more precise, our examples are going to be quasi-linear sections in~$w\Sigma$, 
where the degree of each equation agrees with one of the~$w_i$. The varieties $X$ will have
standard threefold singularities similar to complete intersections in weighted projective
spaces; they have crepant desingularizations~$Y\to X$ by standard theory.

We start by computing the Hilbert series of a weighted flag variety $w\Sigma$ of a given type. 
By numerical considerations, we get candidate degrees for possible Calabi--Yau complete 
intersection families. To prove the existence of a particular family, in particular 
to check that general members of the family only have mild quotient singularities, 
we need equations for the Pl\"ucker style embedding. It turns out that the equations of~$w\Sigma$ 
in the weighted projective space, which are the same as the equations of the straight flag variety
$\Sigma$ in its natural embedding, can be computed relatively easily using computer algebra~\cite{qs}. 

The smooth Calabi--Yau models $Y$ that arise from this method may be new, though it is probably 
difficult to tell. One problem we do know not treat in general is the determination of topological 
invariants such as Betti and Hodge numbers of~$Y$. Some Hodge number calculations for
varieties constructed using a related method are performed in~\cite{anita}, via
explicit birational maps to complete intersections in weighted projective spaces; the Hodge numbers
of such varieties can be computed by standard methods. Such maps are hard to construct in general. 
A better route would be to first compute the Hodge structure of~$w\Sigma$, then deduce 
the invariants of their quasi-linear sections~$X$ and finally their resolutions~$Y$. 
See for example~\cite{baty_cox} for analogous work for hypersurfaces in toric varieties. 
We leave the development of such an approach for future work.

We conclude our paper with the outline of a possible application of our construction: by its 
definition, the weighted flag variety $w\Sigma$ and thus its quasi-linear section $X$ carries 
natural orbi-bundles; these are the analogues of $\Oh(1)$ on (weighted) projective space. 
It is possible that these can be used to construct interesting bundles on the resolution~$Y$ 
which may be relevant in heterotic compactifications. Again, we have no conclusive results. 

\subsection*{Acknowledgements} 
The first author has been supported by a grant from the Higher Education 
Commission (HEC) of Pakistan.

\section{Weighted flag varieties}\label{hssec}

\subsection{The main definition}
We start by recalling the notion of weighted flag variety due to Grojnowski and 
Corti--Reid~\cite{wg}. 
Fix a reductive Lie group $G$ and a highest weight \(\lam \in \Lam_W \), where $\Lam_W$ is
the  weight lattice or lattice of characters of \(G\). Then we have a corresponding parabolic 
subgroup \(P_\la\), well-defined up to conjugation. The quotient  $\Sigma=G/P_\lam$ is a homogeneous
variety called {\em (generalized) flag variety}. 

Let \(\Lam_W^*\) denote the lattice of one parameter subgroups, 
dual to the weight lattice \(\Lam_W\). Choose \(\mu \in \Lam_W^* \) and an integer \(u \in\ZZ\) such 
that \begin{equation}
<w\lam,\mu>+u >0
\label {weights} 
\end{equation}
for all  elements \(w\) of the Weyl group of the Lie group \(G\), where $<,>$ denotes the perfect 
pairing between $\Lam_W$ and $\Lam_W^*$.

Consider the affine cone $\widetilde{\Sigma} \subset  V_{\lambda}$ of the 
embedding $\Sigma \hookrightarrow \PP V_{\lambda}$. There is a $\CC^*$-action on $V_{\lambda}\setminus\{0\}$ 
given by
\[ 
(\varepsilon \in \CC^*) \mapsto ( v \mapsto \varepsilon^u(\mu(\varepsilon)\circ v))
\]
which induces an action on $\widetilde{\Sigma}$. 
The inequality (\(\ref{weights}\)) ensures that all the $\CC^*$-weights on $V_{\lam}$ are positive, 
leading to a well-defined quotient 
\[ w\PP V_{\lam} = V_{\lam} \setminus\{0\} \Big/ \CC^*, \] 
a weighted projective space, and inside it the projective quotient 
\[ w\Sigma = \widetilde{\Sigma} \setminus\{0\} \Big/\CC^* \subset w\PP V_{\lam}.\]
We call $w\Sigma$ a {\em weighted flag variety}.
By definition, \(w \Si\) quasismooth, i.e. its affine cone \(\widetilde \Si\) is nonsingular 
outside its vertex \(\underline 0\). Hence it only has finite quotient singularities. 

The weighted flag variety \(w\Si \) is called {\em well-formed}~\cite{fl}, if no $(n-1)$ 
of weights \(w_i\) have a common factor, and moreover \(w\Si\) does not contain any 
codimension \(c+1\) singular stratum of \(w\PP V_\la\), where \(c\) is the codimension of \(w\Si\). 

\subsection{The Hilbert series of a weighted flag variety}

Consider the embedding \(w\Si \subset w\PP V_{\lam}\). The restriction of the line (orbi)bundle 
of degree one Weil divisors \(\Oh_{w\PP V_{\lam}}(1)\) gives a polarization \(\Oh_{w\Si}(1)\) on 
$w\Sigma$, a \(\QQ\)-ample line orbibundle some tensor power of which is a very ample line bundle. 
Powers of \(\Oh_{w\Sigma}(1)\) have well-defined spaces of sections 
$H^0(w \Sigma, \mathcal{O}_{w\Sigma}(m))$. 
The {\em Hilbert series} of the pair $(w\Si, \Oh_{w\Si}(1))$ is the power series given by
\[
P_{w\Si}(t)= \sum_{m \geq 0} \dim H^0(w \Sigma, \mathcal{O}_{w\Sigma}(m)) t^m.
\]
\begin{theorem}{\rm\cite[Thm. 3.1]{qs}} The Hilbert series $P_{w\Si}(t)$ has the closed form
\begin{equation}
P_{w\Si}(t)=\dfrac{\sum_{w\in W}(-1)^w \dfrac{t^{<w\rho, \mu>}}{(1- t^{<w\lambda,\mu>+u})}}{\sum_{w\in W}(-1)^w t^{<w\rho, \mu>}}.
\label{whhs}
\end{equation}
Here \(\rho\) is the Weyl vector, half the sum of the positive roots of \(G\), and  $(-1)^w=1  \; \mbox{or} -1$ depending on whether $w$ consists of an even or odd number of simple reflections in the Weyl group $W$.
\end{theorem}

The right hand side of~(\ref{whhs}) can be converted into a form
\begin{equation} 
P_{w\Si}(t)=\dfrac{N(t)}{\displaystyle\prod_{\alpha_i \in  \nab(V_\lam)}\left(1-t^{<\alpha_i,\mu>+u}\right)}.
\label{reducedhs}
\end{equation}
Here $\nab(V_\lam)$ denotes the set of weights (understood with multiplicities) appearing in the weight
space decomposition of the representation $V_\lam$; thus the set of weights $w_i = <\alpha_i,\mu>+u$ 
in the denominator agrees with the set of weights of the weighted projective space $w\PP V_{\lambda}$. 
The numerator is a polynomial $N(t)$, the {\em Hilbert numerator}.
Since~\eqref{whhs} involves summing over the Weyl group, it is best to use a computer 
algebra system for explicit computations. 

A well-formed weighted flag variety is {\it projectively Gorenstein}, which means
\begin{enumerate}
\item[i.] \(H^i(w\Si,\Oh_{w\Si}(m))=0 \) for all \(m \text{ and }0 <i< \dim(w\Si);\)  
\item[ii.] the Hilbert numerator $N(t)$ is a palindromic symmetric  
polynomial of degree \(q\), called the {\em adjunction number} of \(w\Si\);  
\item[iii.]  the canonical divisor of \(w\Si\) is given by 
\[K_{w\Si}\sim\Oh_{w\Sigma}\left(q-\sum w_i\right),\] 
where as above, the \(w_i\) are the weights of the projective space \(w\PP V_\la\); the integer
$k=q-\sum w_i$ is called the {\em canonical weight}. \end{enumerate} 

\subsection{Equations of flag varieties} \label{sec:eq}
The flag variety \(\Si=G/P \into \PP V_{\lam}\) is defined by an ideal  \(I =<Q>\)  
of quadratic equations generating a linear subspace \(Q \subset Z=S^2 V^*_{\lam}  \)
of the second symmetric power of the contragradient representation \(V^*_{\lam}\). 
The $G$-representation $Z$ has a decomposition 
\begin{displaymath}
Z=V_{2\nu}\oplus V_1 \oplus\cdots \oplus V_n
\end{displaymath}  
into irreducible direct summands, with \(\nu\) being the highest weight of the 
representation \(V^*_{\lam}\). As discussed in \cite[2.1]{rudakov}, 
the subspace \(Q\) in fact consists of all the summands except~\(V_{2\nu}\). 
The equations of $w\Sigma$ can be readily computed from this information using computer algebra~\cite{qs}.

\subsection{Constructing Calabi--Yau threefolds}
We recall the different steps in the construction of Calabi--Yau threefolds as
quasi-linear sections of weighted flag varieties.

\begin{enumerate}\item \textbf{Choose embedding.} We choose a reductive Lie group 
\(G\) and a \(G\)-representation \(V_{\lam}\) of dimension $n$ 
with highest weight \(\lam\). 
We get a straight flag variety $\Si = G/P_\lam \into \PP V_{\lam}$ of 
computable dimension $d$ and codimension~$c=n-1-d$. We choose \(\mu \in \Lam_W^*\) 
and \(u \in \ZZ\) to get an embedding \(w\Si \into w\PP V_{\lam}=\PP^{n-1}[<\al_i,\mu>+u]\), with 
\( \al_i \in \nab(V_\lam)\) the weights of the representation \(V_\la\).
The equations, the Hilbert series and the canonical class of $w\Si\subset w\PP$ 
can be found as described above.

\item \textbf{Take threefold Calabi--Yau section of \(w\Si\).} 
We take a quasi-linear complete intersection
\begin{displaymath}
X= w\Si \cap (w_{i_1})\cap \cdots \cap (w_{i_l})
\end{displaymath} 
of $l$ generic hypersurfaces of degrees equal to some of the weights $w_i$. 
We choose values so that \(\dim(X)=d-l=3\) and $k+\sum_{j=1}^l w_{i_j}=0$, thus \(K_X\sim \Oh_X\).  
After re-labelling the weights, 
this gives an embedding \(X \into \PP^s[w_0,\cdots,w_s]\), with \(s=n-l-1\), of codimension $c$,
polarized by the ample $\QQ$-Cartier divisor $D$ with $\Oh_X(D)=\Oh_{w\Sigma}(1)|_X$. 
More generally, as in~\cite{wg}, we can take complete intersections inside projective cones
over $w\Si$, adding weight one variables to the coordinate ring which are not involved in any 
relation.

\item \textbf{Check singularities.} We are interested in quasi-smooth 
Calabi--Yau threefolds, subvarieties of $w\Sigma$ all of whose  singularities are induced by the 
weights of \(\PP^s[w_i]\). Singular strata $S$ of \(\PP^s[w_i]\) correspond to sets of weights 
\(w_{i_0},\cdots,w_{i_p}\) with \[\gcd(w_{i_0},\cdots,w_{i_p})=r\]  
non-trivial. If the intersection \(X\cap S\) is non-empty, it has to be a singular point $P\in X$ or 
a curve $C\subset X$ of quotient singularities, and we need to find local coordinates in 
neighbourhood of points of $P$ respectively $C$ to check the local
transversal structure. Since we are interested in Calabi--Yau varieties which admit crepant 
resolutions, singular points $P$ have to be quotient singularities of the 
form $\frac{1}{r}(a, b, c)$ with $a+b+c$ divisible by $r$, whereas the transversal singularity 
along a singular curve $C$ has to be of the form $\frac{1}{r}(a, r-a)$ of type $A_{r-1}$. 

\item \textbf{Find projective invariants and check consistency.} The orbifold Riemann--Roch formula 
of~\cite[Section 3]{anita} determines the Hilbert series of a polarized Calabi--Yau 
threefold $(X,D)$ with quotient singularities in terms of the projective invariants $D^3$ 
and $D.c_2(X)$, as well as for each curve, the degree $\deg D|_C$ of the polarization, 
and an extra invariant $\gamma_C$ related to the normal bundle of $C$ in $X$. Using the Riemann--Roch
formula, we can determine the invariants of a given family from the first few values of $h^0(nD)$, 
and verify that the same Hilbert series can be recovered.
\end{enumerate}

\subsection{Explicit examples}

In the next two sections, we find families of Calabi--Yau threefolds admitting crepant resolutions
using this programme. We illustrate the method using two embeddings, corresponding to the Lie
groups of type $G_2$ and $A_5$, leading to Calabi--Yau families of codimension $8$, 
respectively $6$. Further examples for the Lie groups of type \(C_{3}\) and  \(A_3\), in codimensions 
7 and 9, will be discussed in the forthcoming~\cite{Q}.

\section{The codimension eight weighted flag variety}\label{g2sec}

\subsection{Generalities} Consider the simple Lie group of type $G_2$. Denote by $\al_1, \al_2 \in \Lam_W$ 
a pair of simple roots of the root system $\nabla$ of $G_2$, 
taking $\al_1$ to be the short simple root and $\al_2$ the long one.
The fundamental weights are $ \omega_1=2\al_1+\al_2 $ and $ \omega_2=3\al_1+2\al_2$. 
The sum of the fundamental weights, which is equal to half the sum of the positive roots, 
is $\rho=5\al_1+3\al_2 $. We partition the set of roots into long and short roots as 
$\nabla=\nabla_l \cup \nabla_s \subset \Lam_W $. 
Let ${\{\be_1,\be_2\}}$ be the basis of  the lattice $\Lam_W^*$ 
dual to \(\{\al_1 , \al_2\}\).

We consider the  $G_2$-representation with highest weight  $\lambda=\om_2=3\al_1+2\al_2$. 
The dimension of $ V_{\lam} $ is 14 \cite[Chapter 22]{harris}. The homogeneous variety 
$\Sigma \subset \PP V_{\lam}$ is five dimensional, so we have an embedding 
\(\Si^5 \into \PP^{13}\) of codimension 8. To work out the weighted version in this case, 
take $\mu=a\be_1+b\be_2 \in \Lam_W^*$ and \(u\in \ZZ\).

\begin{prop} The Hilbert series of the codimension eight weighted \(G_2\) flag variety is given by \begin{equation}  P_{w\Sigma}(t)= \dfrac{ 1- \left( 4+ 2 \sum_{\alpha \in \nabla_s}t^{<\alpha,\mu>} + \sum_{\alpha \in \nabla_s}t^{2< \alpha,\mu >}+\sum_{\alpha \in \nabla_l}t^{< \alpha,\mu >}\right) t^{2u}+\cdots+t^{11u} } {(1-t^u)^2 \prod_{\alpha \in \nabla}\left( 1-t^{<\alpha,\mu>+u}\right) }.\end{equation}  Moreover, if \(w\Si\) is well-formed, then the canonical bundle is $K_{w\Si}\sim\Oh_{w\Si}(-3u)$. \label{hsg2long} \end{prop}
\begin{proof}
\end{proof}

The Hilbert series of the straight flag variety \(\Si\into \PP^{13}\) can be computed to be 
$$P_{\Sigma}(t)= \dfrac{1-28t^2+105t^3-\cdots+105t^8-28t^9-t^{11}}{(1-t)^{14}}.$$
The image is defined by 28 quadratic equations, listed in the Appendix of~\cite{qs}.

\subsection{Examples}

\begin{example}\label{u31} Consider the following initial data.
\begin{itemize}
\item Input: $\mu=(-1,1), u=3$.
\item Pl\"ucker embedding: $w\Si\subset \PP^{13}[1,2^4,3^4,4^4,5]$.
\item Hilbert numerator: $1-3 t^4-6 t^5-8 t^6+6 t^7+21 t^8+\ldots+6 t^{26}-8 t^{27}-6 t^{28}-3 t^{29}+t^{33}$.
\item Canonical divisor: $ K_{w\Si}\sim \Oh_{w\Si}(33-\sum_i w_i)= \Oh(-9)$, as $w\Si$ is well-formed.
\item Variables on weighted projective space together with their weights $x_i$:
\[
 \renewcommand{\arraystretch}{1.5}
 \begin{array}{ccccccccccccccc}
{\rm Variables} & x_1   & x_2   & x_3   & x_4&x_5&x_6&x_7&x_8&x_9&x_{10}&x_{11}&x_{12}&x_{13}&x_{14}\\
{\rm Weights }  &2&4&3&2&1&2&4&2&3&4&5&4&3&3
 \end{array}\]
The reason for the curious ordering of the variables is that these variables are exactly those appearing
in the defining equations of this weighted flag variety given in \cite[Appendix]{qs}.
\end{itemize} 
Consider the threefold quasilinear section 
\[X= w\Si \cap \{f_4(x_i)=0\}\cap\{g_5(x_i)=0\} \subset \PP^{11}[1,2^4,3^4,4^3], \] 
where the intersection is taken with general forms $f_4,g_5$ of degrees four and five respectively. 
The canonical divisor class of \(X\) is
\[K_X\sim\Oh_X(-9+(5+4)) = \Oh_X.\] 
To determine the singularities of the general threefold $X$, we need to consider sets of variables
whose weights have a greatest common divisor greater than one.
\begin{itemize}
\item $1/4$ singularities: this singular stratum is defined by setting those variables to zero whose degrees are 
not divisible by~$4$. We also have the equations of~\cite[Appendix]{qs}; only (A5), (A23) and (A24) from that list
survive to give
\[S= \left\{\begin{array}{c} \frac{1}{9}x_7x_{10} + x_2 x_{12} =0 \\ -\frac{1}{3}x_{10}^2 + x_7 x_{12} = 0 \\ \frac{1}{3}  x_7^2 + x_2 x_{10} = 0 \end{array}\right\}\subset \PP^3_{x_2, x_7, x_{10}, x_{12}}.\]
In this case, it is easy to see by hand 
(or certainly using Macaulay) that $S\subset\PP^3$ is in fact a twisted cubic curve isomorphic to
$\PP^1$. We then need to intersect this with the
general $X$; the quintic equation will not give anything new, since $x_2, x_7, x_{10}, x_{12}$
are degree~4 variables, but the quartic equation will give a linear relation between them. Thus
$S\cap X$ consists of three points, the three points of $1/4$ singularities. A little
further work gives that they are all of type \(\dfrac{1}{4}(3,3,2).\)
% (A23) x_7x_{12}=0 (A25) x_7^2=0 
\item $1/3$ singularities: the general $X$ does not intersect this singular stratum, the equations
from~\cite[Appendix]{qs} in the degree three variables give the empty locus; this is easiest to check
by Macaulay.
\item $1/2$ singularities: the intersection of $X$ with this singular stratum is
a rational curve $C\subset X$ containing the $1/4$ singular points; 
again, Macaulay computes this without difficulty. 
At each other point of the curve we can check that the 
transverse singularity is \(\dfrac{1}{2}(1,1).\)
\end{itemize}
Thus \((X,D)\)  is a Calabi--Yau threefold with three singular points of type \(\dfrac{1}{4}(3,3,2) \) and a rational curve \(C\) of singularities of type \(\dfrac{1}{2}(1,1)\) containing them. Comparing with the orbifold Riemann--Roch formula of~\cite[Section 3]{anita}, feeding in the first few known values of $h^0(X, nD)$ from the Hilbert series gives that the projective invariants of this family are 
\[ D^3= \dfrac{9}{8}, \ D.c_2(X)=21, \ \deg D\rvert_{C}= \dfrac{9}{4}, \ \gamma_C=1.\]
\end{example}

\begin{example} In this example, we consider the same initial data as in Example \ref{u31}. 
To construct a new family of Calabi--Yau threefolds, we take a projective cone over \(w\Si_{}\). 
Therefore we get the embedding \[\Cc w\Si_{} \subset \PP^{14}[1^2,2^4,3^4,4^4,5].\] 
The canonical divisor class of \(\Cc w \Si_{}\) is \(K_{\Cc w \Si}\sim\Oh_{\Cc w \Si}(-10).\) Consider the 
threefold quasilinear section 
\[X= \Cc w\Si_{} \cap (5)\cap(3) \cap (2) \subset \PP^{11}[1^2,2^{3},3^3,4^4,5]\]
with $K_X\sim \Oh_X$; brackets $(w_i)$ denote a general hypersurface of degree $w_i$.
 
\begin{itemize}
\item $1/4$ singularities: since there is no quartic equation this time, the whole twisted cubic curve 
\(C\subset \PP^3[x_2,x_7,x_{10},x_{12}]\), found above, is contained in the general $X$, and is
a rational curve of singularities of type \(\dfrac{1}{4}(1,3)\).
\item $1/3$ singularities: the general \(X\) does not intersect this singular stratum.
\item $1/2$ singularities: the intersection of \(X\) with this singular strata defines a further rational 
curve \(E\) of singularities. On each point of the curve we check that local transverse parameters have 
odd weight. Therefore \(E\) is a curve of type \(\dfrac{1}{2}(1,1).\)
\end{itemize}
Thus \((X,D)\)  is a Calabi--Yau threefold with two disjoint rational curves of 
singularities \(C\) and \(E\) of type \(\dfrac{1}{4}(1,3)\) and \(\dfrac{1}{2}(1,1)\) respectively. 
The rest of the invariants of this family are
\[ D^3= \dfrac{27}{16},\ D.c_2(X)=21, \ \deg D\rvert_{C}= \dfrac{3}{4}, \ \gamma_C=2, \ \deg D\rvert_{E}= \dfrac{3}{4}, \ \gamma_E=1.\]
\end{example}

\begin{example} The next example is obtained by a slight generalization of the method 
described so far. The computation of the canonical class \(K_{w\Si_{}}\), as the basic line
bundle $\Oh_{w\Sigma}(1)$ raised to the power equal to the difference of the adjunction 
number and the sum of the weights on \(w\PP^n \), only works if \(w\Si_{}\) is well-formed. 
In this example, we will make our ambient weighted homogeneous variety not well-formed. 
We then turn it into a well-formed variety by taking projective cones over it. 
We finally take a quasilinear section to construct a Calabi--Yau threefold \((X,D).\)
\begin{itemize}
\item Input: $ \mu=(0,0)$, $u=2$.
\item Pl\"ucker embedding: $ w\Si_{} \subset \PP^{13}[2^{14}]$, not well-formed.
\item Hilbert Numerator: $ 1-28 t^4+105 t^6-162 t^8+84 t^{10}+84 t^{12}-162 t^{14}+105 t^{16}-28 t^{18}+t^{22}$.
\end{itemize}
We take a double projective cone over \(w\Si\), by introducing two new variables \(x_{15}\) and \(x_{16}\) of weight one, which are not involved in any of the defining equations of \(w\Si_{}\). 
We get a seven-dimensional well-formed and quasismooth variety
\[\Cc \Cc w\Si \subset \PP^{15}[1^2,2^{14}]\] with canonical class 
\(K_{\Cc \Cc w\Si}\sim\Oh_{\Cc \Cc w\Si}(-8)\).

Consider the threefold quasilinear section 
\[X=\Cc \Cc w\Si_{} \cap (2)^4 \subset \PP^{11}[1^2,2^{10}].\] 
The canonical class \(K_X\) becomes trivial. Since \(w\Si_{}\) is a five dimensional variety,
and we are taking a complete intersection with four generic hypersurfaces of degree two inside 
\(\PP^{15}[1^2,2^{14}]\), the singular locus defined by weight two variables defines a curve 
in \(\PP^{11}[1^2,2^{10}]\). 
Thus \((X,D)\)  is a Calabi--Yau threefold with a curve of singularities of 
type \(\dfrac{1}{2}(1,1).\) The rest of the invariants of \((X,D)\) are given as follows.
\[D^3= \dfrac{9}{2}, \ D.c_2(X)=42, \ \deg D\rvert_{C}= 9, \ \gamma_C=1.\]
\end{example}

\begin{example} Our final initial data in this section consists of the following.
\begin{itemize}
\item Input: $ \mu=(-1,1)$, $u=5$.
\item  Pl\"ucker embedding: $ w\Si_{} \subset \PP^{13}[3,4^4,5^4,6^4,7]$.
\item Hilbert Numerator: $ 1-3 t^8-6 t^9-10 t^{10}-6 t^{11}-t^{12}+12 t^{13}+\ldots+t^{55}.$
\item Canonical class: $ K_{w\Si_{}}\sim \Oh_{w\Si_{}}(-15) $, as $ w\Si_{} $ is well-formed.
\end{itemize} We take a projective cone over \(w\Si_{}\) to get the embedding
\[\Cc w \Si_{} \subset \PP^{14}[1,3,4^4,5^4,6^4,7]\] 
with \(K_{\Cc w \Si_{}}\sim \Oh_{\Cc w \Si_{}}(-16).\) We take a complete intersection inside \(\Cc w\Si_{}\), with three general forms 
of degree seven, five and four in \(w\PP^{14}\). Therefore we get a threefold 
\[X= \Cc w \Si_{} \cap (7) \cap (5) \cap (4) \into\PP^{11}[1,3,4^3,5^3,6^4],\]
with trivial canonical divisor class. To work out the singularities, we work through the singular strata to 
find that \((X,D)\) is a polarised Calabi--Yau threefold containing three dissident singular  points of type 
\(\dfrac{1}{4}(1,1,2) \), a rational curve of singularities \(C\) of type \(\dfrac{1}{6}(1,5)\) containing them, 
and a further isolated singular point of type \(\dfrac{1}{3}(1,1,1) \). The rest of the invariants are  
 \[ D^3= \dfrac{5}{24},\
  D.c_2(X)=17,  \ \deg D\rvert_{C}= \dfrac{5}{4}, \ \gamma_C=9.\]
\end{example}

\section{The codimension 6 weighted Grassmannian variety}
\label{g26sec}

\subsection{The weighted flag variety} 
We take $G$ to be the reductive Lie group of type $\GL(6,\CC)$. 
The five simple roots are $\al_i=e_i-e_{i+1}\in \Lam_W$, the weight lattice with basis $e_1, \ldots, e_6$.
The Weyl vector can be taken to be
\begin{displaymath}\rho=5e_1+4e_2+3e_3+2e_4+e_5.\end{displaymath}
Consider the irreducible $G$-representation $V_{\lam}$, with $\lam=e_1+e_2$. Then $V_{\lam}$ is 
15-dimensional, and all of the weights appear with multiplicity one. 
The highest weight orbit space $\Si=G/P_{\lam} \subset \PP V_{\lam}=\PP^{14}$ is eight dimensional.  
This flag variety can be identified with the Grassmannian 
of 2-planes in a 6-dimensional vector space, a codimension 6 variety
\[ 
\Si^8= \Gr(2,6) \hookrightarrow \PP V_{\lam}= \PP^{14}.
 \]

Let $\left\lbrace f_i,  1\leq i \leq 6 \right\rbrace $ be the dual basis of the dual lattice $\Lam_W^*$. 
We choose \[  \mu=  \sum_{i=1}^6a_if_i \in \Lam_W^*,\] and $u \in \ZZ$, 
to get the weighted version of $\Gr(2,6)$,
\[ w\Si(\mu,u)=w\Gr(2,6)_{(\mu,u)} \into w\PP^{14}. \]  
The set of weights on our projective space is $ \left\lbrace <\la_i,\mu> + u \right\rbrace  $, where $ \la_i $ are weights appearing in the \(G\)-representation \(V_{\lam}\). As a convention we will write an element of dual lattice as row vector, i.e. $\mu=(a_{1},a_2,\cdots,a_6).$

We expand the formula (\ref{whhs}) for the given values of \(\lam,\mu\) to get the following formula for the Hilbert series of \(w\Gr(2,6).\) 
 \[P_{w\Gr(2,6)}(t)=\dfrac{1-Q_1(t)t^{2u}+Q_2(t)t^{3u}-Q_3(t)t^{4u}-Q_4(t)t^{5u}+Q_5(t)t^{6u}-Q_6(t)t^{7u}+
 t^{3s+9u}}{\prod_{1\leq i<j\leq 6} (1-t^{a_i+a_j+u})}.\]
Here 
\[
Q_1(t)=\sum_{1\leq i <j\leq 6}t^{s-(a_i+a_j)}
,\ \ \ Q_2(t)=\sum_{1\leq (i,j)\leq 6}t^{s+(a_i-a_j)}-t^s,\]\[Q_3(t)=\sum_{1\leq i \leq j\leq 6}t^{s+(a_i+a_j)},\ \ \ Q_4(t)=\sum_{1\leq i \leq j\leq 6}t^{2s-(a_i+a_j)},\]
\[Q_5(t)=\sum_{1\leq (i , j)\leq 6}t^{2s+(a_i-a_j)}-t^{2s},\ \ \ Q_6(t)=\sum_{1\leq i \leq j\leq 6}t^{2s+(a_i+a_j)}.\]
In particular, if \(w\Gr(2,6) \into \PP^{14}[<w_i,\mu>+u]\) is well-formed, then its canonical bundle is 
\(K_{w\Gr(2,6)}\sim\Oh_{w\Gr(2,6)}(-2s-6u)\), with \(s=\displaystyle\sum_{i=1}^6a_i.\)

The defining equations for $ \Gr(2,6)\subset\PP^{14}$ are well known to be the $4 \times 4$ Pfaffians 
obtained by deleting two rows and the corresponding columns of the $6\times 6$ skew symmetric matrix
\begin{equation}\label{eqgr26} A=
\begin{bmatrix}
0 & x_{1} &x_{2}  &x_{3}  &x_{4}&x_5   \\
   & 0         & x_{6} & x_{7} & x_{8}&x_{9}   \\
    &        & 0          & x_{10} & x_{11}&x_{12}   \\
     &       &             & 0          & x_{13}&x_{14}\\
      &&&&                                0  &x_{15}\\
&&&&&                                          0                          
\end{bmatrix} .
\end{equation}

\subsection{Examples}

\begin{example}
Consider the following data.
\label{expgr261}
\begin{itemize}
\item Input: $\mu=(2,1,0,0,-1,-2), u=4$.
\item Pl\"ucker embedding: $ w \Gr(2,6) \subset \PP^{14}[1,2^2,3^3,4^3,5^3,6^2,7]$.
\item Hilbert Numerator: $1-t^5-2 t^6-3 t^7-2 t^8-t^9+\cdots +t^{36}$.
\item Canonical class: $K_{w\Gr(2,6)}\sim \Oh_{w\Gr(2,6)}(-24)$.
\end{itemize} 
Consider the three-fold quasi-linear section
$$X= w\Gr(2,6) \cap (7)\cap (6)\cap (5)\cap (4)\cap (2) \subset \PP^9[1,2,3^3,4^2,5^2,6].$$
Then $K_X$ is trivial, and $X$ is a Calabi--Yau 3-fold  with a singular point of type $\dfrac{1}{6}(5,4,3)$, 
lying on the intersection of two curves, $C$ of type $\dfrac{1}{3}(1,2)$ and $E$ of type $\dfrac{1}{2}(1,1)$. 
There is an additional isolated singular point of type  $\dfrac{1}{5}(4,3,3)$. The rest of the invariants of 
this variety are
\[ D^3= \dfrac{11}{30}, \ D.c_2(X)=\dfrac{68}{5},\  \deg D\rvert_{C}= \dfrac{1}{3}, \ \gamma_C=\dfrac{ -15}{2},\ \deg D\rvert_{E}= \dfrac{1}{2}, \ \gamma_E= 1.\]
\end{example} 

\begin{example} We take the following. 
\begin{itemize}
\item Input: $\mu=(2,1,1,1,1,0) $, $u=0$.
\item Pl\"ucker embedding: $w\Gr(2,6) \subset \PP^{14}[1^4,2^7,3^4]$.
\item Hilbert Numerator: $1-4 t^3-6 t^4+4 t^5+\cdots+t^{18}$.
\item Canonical class: $K_{w\Gr(2,6)}\sim \Oh_{w\Gr(2,6)}(-12)$, as $w\Si$ is well-formed.
\end{itemize}
Consider the quasilinear section \[ X=w\Gr(2,6) \cap (3)^2\cap (2)^3 \subset \PP^9[1^4,2^4,3^2] ,\] then \[ K_X= \Oh_X(-12+(2 \times 3+3\times 2)) =\Oh_X.\]
The variety $(X,D)$ is a well-formed and quasismooth Calabi--Yau 3-fold. Its singularities consist of two 
rational curves $C$ and $E$ of singularities of type  $\dfrac{1}{3}(1,2)$ and $\dfrac{1}{2}(1,1)$ respectively. 
The rest of the invariants are
\[ D^3= \dfrac{97}{18}, \ D.c_2(X)=42,  \ \deg D\rvert_{C}= \dfrac{1}{3},\ \gamma_C= 2,  \ \deg D\rvert_{E}= 1, \ \gamma_E= 1.
\]
\end{example}

\section{Tautological (orbi)bundles}

\subsection{The classical story}
Let \(\Si=G/P\) be a flag variety. A representation $V$ of the parabolic subgroup \(P\) gives
rise to a vector bundle \(\sE\) on \(\Si\) as follows:
\[\begin{diagram} \node{\sE  = G \times_{P} V} \arrow{s} \\ \node{\Sigma=G/P} \end{diagram}\] 
In other words, the total space of \(\sE\) consists of  pairs \((g,e) \in G \times V\) modulo the equivalence 
\[(gp,e) \sim (g,pe)  \text{ for } p \in P.\]
The fiber of $\sE$ over each point \(\Si\) is isomorphic to the vector space underlying \(V\). 

\begin{example} The simplest example is $\Sigma=\PP^{n-1}$, a homogeneous variety $G/P$ with $G=\GL(n)$ and
$P$ the parabolic subgroup consisting of matrices of 
the form \[A = \left(\begin{array}{cccc} \alpha & * &\cdots & * \\ 0 \\ \vdots & &  B \\ 0 & \end{array}\right).\]
We obtain a one-dimensional representation of $P$ by mapping $A$ to $\alpha$. The associated line bundle is 
just the tautological line bundle on $\PP^{n-1}$, the dual of the hyperplane bundle $\Oh_{\PP^{n-1}}(1)$. 
\end{example}

\begin{example} More generally, consider $\Sigma=\Gr(k,n)$, the Grassmannian of $k$-planes in 
$\CC^n$. Then $G=\GL(n)$  and the corresponding parabolic is the subgroup of matrices of the form
\[A = \left(\begin{array}{cccc} B_1 & * \\ 0 & B_2 \end{array}\right),\]
with $B_1, B_2$ of size $k\times k$ and $(n-k)\times(n-k)$ respectively. The representations of $P$ defined
by $A\mapsto B_1$, respectively $A\mapsto B_2$ give the standard tautological sub-, and quotient bundles~$\sS$ 
and~$\sQ$ on the Grassmannian $\Gr(k,n)$, fitting into the exact sequence
\[0\to\sS\to\Oh_{\Gr(k,n)}^{\oplus n} \to \sQ\to 0.\] 
\end{example}

\begin{example} Finally consider the $G_2$-variety $\Sigma=G/P$ studied in Section~\ref{g2sec}. 
The smallest representations of the corresponding $P$ have dimensions 2 and 5. The corresponding tautological
bundles are easiest to describe using an embedding $\Sigma\hookrightarrow \Gr(2,7)$, mapping the 
$G_2$ flag variety into the
Grassmannian of 2-planes in a 7-dimensional vector space, the space ${\rm Im}\,\OO$ of imaginary octonions. 
Then the tautological bundles on the $G_2$-variety $\Sigma$ 
are the restrictions of the tautological sub- and quotient-bundle from $\Gr(2,7)$.
\end{example}

\subsection{Orbi-bundles on Calabi--Yau sections}

Recall that weighted flag varieties are constructed by first considering the $\CC^*$-covering
$\widetilde \Si\setminus \{0\}\to \Sigma$, and then dividing $\widetilde \Si\setminus \{0\}$ by a different
$\CC^*$-action given by the weights. A tautological vector bundle $\sE$ on $\Si$ pulls back to a
vector bundle $\widetilde\sE$ on $\widetilde \Si\setminus \{0\}$. This can then can be pushed forward 
to a weighted flag variety $w\Si$ along the quotient map $\widetilde \Si\setminus \{0\}\to w\Si$. 
Because of the finite stabilizers that exist under this second action, the resulting object $w\sE$
is not a vector bundle, but an orbibundle~\cite[Section 4.2]{bg}, which trivializes on local orbifold covers 
with compatible transition maps. If \(X\) is a Calabi--Yau threefold inside \(w\Si\), then we can define an 
orbi-bundle on \(X\) by restricting \(w\sE\)  to \(X\). 

In the constructions of Sections~\ref{g2sec}-\ref{g26sec}, 
the Calabi--Yau sections therefore carry possibly interesting 
orbi-bundles of ranks 2 and 5, respectively 4. We have not investigated the question whether these orbi-bundles
can be pulled back to vector bundles on a resolution $Y\to X$, but this seems to be of some interest. 
If so, stability properties of the resulting vector bundles may deserve some investigation, in view
of their possible use in heterotic model building~\cite{gsw, donagi}.

\noindent {\sc Mathematical Institute, University of Oxford.}

\noindent {\sc 24-29 St Giles', Oxford, OX1 3LB, United Kingdom.}

\vspace{0.1in}

\noindent {\tt qureshi@maths.ox.ac.uk}

\noindent {\tt szendroi@maths.ox.ac.uk}

\end{document}